\documentclass[12pt]{amsart}
\usepackage{graphicx}
\usepackage{amsmath}
\usepackage{amscd}
\usepackage{graphics}
\usepackage{latexsym,multirow,dcpic,pictexwd}
\theoremstyle{definition}
\newtheorem{thm}{Theorem}[section]
\newtheorem{cor}[thm]{Corollary}

\newtheorem{prop}[thm]{Proposition}

\newtheorem{theo}[thm]{Theorem}

\newtheorem{rem}[thm]{Remark}

\numberwithin{equation}{section}

\newcommand{\spinc}{$\mathrm{Spin}^c\;$}

\begin{document}
\title[]{Heegaard Floer homology of Some Mazur Type Manifolds}%
\author{ Selman Akbulut, \c{C}a\u{g}ri Karakurt}%
\thanks{The first named author is partially supported by NSF grant DMS 9971440}
\thanks{ The second named author is  supported by a Simons fellowship.}
\thanks{Both authors are partially supported by NSF Focused Research Grant DMS-1065955.}
\address{Department of Mathematics Michigan State University \\
         East Lansing 48824 MI, USA }%
\email{ akbulut@math.msu.edu}
\address{Department of Mathematics
The University of Texas at Austin\\
2515 Speedway, RLM 8.100
Austin, TX 78712 }%
\email{ karakurt@math.utexas.edu}%

\date{\today}
\begin{abstract}
We show that an infinite family of contractible $4$--manifolds have the same boundary as a special type of plumbing. Consequently their Ozsv\'ath--Szab\'o invariants can be calculated algorithmically. We run this algorithm for the first few members of the family and list the resulting Heegaard--Floer homologies. We also show that the rank of the Heegaard--Floer homology can get arbitrary large values in this family by using its relation with the Casson invariant. For comparison, we list the ranks of Floer homologies of all the examples of Briekorn spheres that are known to bound contractible manifolds.
   
\end{abstract}


%
\maketitle

\section{Mazur Type Manifolds}
A $4$--manifold is said to be Mazur type if it admits a handle decomposition consisting of a single handle at each index $\{0,1,2\}$ where the $2$-handle  is attached in such a way that it cancels the $1$--handle algebraically, \cite{AK}. For $n\in \mathbb{N}$ and $k\in\mathbb{Z}$, let $W_n(k)$ be the Mazur type manifold whose handlebody picture is as drawn in Figure \ref{fig:Cork}.  The main purpose of this note is to calculate the Heegaard Floer homology group $HF^+(\partial W_n(k))$ for every $k$ and $n$. We are especially interested in the case where $k=0$, because the manifolds $W_n(0)$ are the corks which have been extensively used in the construction exotic smooth structures on $4$-manifolds, \cite{AY}, \cite{AKa}. Understanding the Floer homology of their boundary would be the first step in finding out how the smooth 4-manifold invariants change under surgery along these manifolds. Another related family will be discussed in section \ref{sec:Bries}. 
\begin{figure}[h]
	\includegraphics[width=0.35\textwidth]{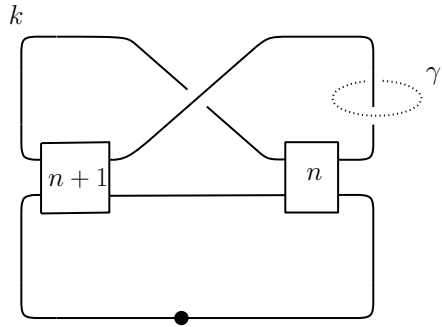}
		\caption{The Mazur type manifold $W_n(k)$.}
	\label{fig:Cork}
\end{figure}

We first review some standard facts from \cite{OS4}, \cite{OS5}, and \cite{OS2}.  Heegaard Floer homology of any $3$-manifold $Y$  is a $\mathbb{Z}/2\mathbb{Z}$ graded abelian group which splits as a direct sum with respect to the \spinc structures on $Y$.  For any \spinc structure $\mathfrak{t}$, the corresponding group $HF^+(Y,\mathfrak{t})$ admits an endomorphism $U$ which preserves $\mathbb{Z}/2\mathbb{Z}$ grading.  This endomorphism equips Heegaard Floer homology with the structure of a $\mathbb{Z}[U]$ module. When a \spinc structure is torsion, the associated component of Heegaard Floer homology admits an absolute $\mathbb{Q}$--grading. In this case $U$ has degree $-2$. When $Y$ is an integral homology sphere, there is a single \spinc component where the absolute  grading in fact takes integer values. In this case the parity of this $\mathbb{Z}$--grading agrees with the  $\mathbb{Z}/2\mathbb{Z}$ grading, and the following splitting occurs:

\begin{equation}\label{eqn:str}
HF^+(Y)=\mathcal{T}^+_{(d)}\oplus HF^{\mathrm{red}}(Y)
\end{equation}

\noindent where $HF^{\mathrm{red}}(Y)$ is a finitely generated abelian group, and $\mathcal{T}^+_{(d)}$ is a copy of $\mathbb{Z}[U,U^{-1}]/ U \mathbb{Z}[U] $ graded so that the minimal degree element has degree $d$. This minimal degree is called the correction term of $Y$. Moreover, the Euler characteristic of the reduced part satisfies

$$\chi (HF^{\mathrm{red}}(Y))= \lambda (Y)+ d(Y)/2$$

\noindent where $\lambda (Y)$ is the Casson invariant of $Y$, \cite{AM}.

It is evident from the definition that any Mazur type manifold $W$ is contractible, so its boundary is an integral homology sphere. Therefore $HF^+(\partial W)$ has the structure mentioned above. In addition to that, we have the following restrictions.

\begin{prop}\label{prop:str}
Suppose that an integral homology sphere bounds a contractible manifold. Then its correction term is zero and its Casson invariant is even.   
\end{prop}

\begin{proof}
The correction term of a homology sphere  is invariant under rational homology cobordisms. Any contractible manifold $W$ is an integral homology ball. We can also regard it as an integral homology cobordism from  $S^3$ to $\partial W$, by removing a small ball from it. Since the correction term of the $3$-sphere is zero,  we have $d(\partial W)=0$. The second statement is a consequence of the fact that the Mod $2$ reduction of the Casson invariant of any integral homology sphere is in fact the Rohlin invariant, which can be computed from the signature of the Spin $4$--manifold it bounds. Since the signature of any contractible manifold is zero the second part follows.    
\end{proof}

Next we  prove that up to a grading change, the Heegaard Floer homology is independent of the framing of the $2$-handle of a Mazur type manifold. See \cite{S}, for an analogous statement in Instanton Floer homology.

\begin{prop}\label{indep}
Let $W(k)$ be a Mazur type manifold, $k$  being the framing  of the $2$-handle. Then for any $k,k'\in\mathbb{Z}$, there is a $U$ equivariant  isomorphism between $HF^+(\partial W(k))$ and  $HF^+(\partial W(k'))$ which preserves $\mathbb{Z}/2\mathbb{Z}$ gradings. Moreover, this isomorphism also preserves absolute $\mathbb{Q}$ gradings if the meridional circle $\gamma$ of the $2$-handle bounds a genus one Seifert surface in $\partial W(k)$ for every $k$.
\end{prop}

\begin{proof}
By Equation \ref{eqn:str} and proposition \ref{prop:str}, it suffices to show  that $HF^{\mathrm{red}}(W(k))$ is independent of $k$. To see this, we use the surgery exact triangle \cite{OS5}. Observe that $-1$ framed surgery on the curve $\gamma$ inside $\partial W_n(k)$ gives $\partial W(k+1)$ and $0$ surgery along the same curve gives $S^1\times S^2$. The Heegaard Floer homology of $S^1\times S^2$, which is supported in the torsion \spinc  structure $\mathfrak{t}_0$, is given by $HF^+(S^1\times S^2,\mathfrak{t}_0)=\mathcal T_{(1/2)}\oplus \mathcal T_{(-1/2) }$. The surgery exact triangle can be read as follows:

\begin{equation}\label{eq:exact}
\begindc{\commdiag}[10]
\obj(5,3)[N1]{$\cdots$}
\obj(13,3)[A]{$HF_p^+(\partial W(k+1))$}
\obj(25,3)[B]{$HF_{p-1/2}^+(S^1\times S^2,\mathfrak{t}_0)$}
\obj(37,3)[C]{$HF_{p-1}^+(\partial W(k))$}
\obj(45,3)[N2]{$\cdots$}
\mor{N1}{A}{$f_3$}[1,0]
\mor{A}{B}{$f_1$}[1,0]
\mor{B}{C}{$f_2$}[1,0]
\mor{C}{N2}{$f_3$}[1,0]
\enddc
\end{equation}

The homomorphisms $f_1$ and $f_2$ are homogeneous of degree $-1/2$. To see this recall that each one of $f_1$, $f_2$, and $f_3$ is induced from a $4$--dimensional cobordism. For example, $f_1$ corrpesponds to a $2$-handle attacthment along $\gamma$ with framing $0$ giving a cobordism $C$ with boundary $\partial C= -\partial W(k+1)\sqcup S^1\times S^2$. Now, every \spinc structure on $S^1\times S^2$ has a unique extension on $C$.   Therefore $f_1=\pm F_{C,\mathfrak{s}_0}$ where $\mathfrak{s}_0$ is the torsion \spinc structure on $C$, and $F_{C,\mathfrak{s}_0}$ is the homomorphism induced by the pair $(C,\mathfrak{s}_0)$. Then the  degree  of $f_1$ is given by $(c_1(\mathfrak{s}_0)^2-3\sigma(C)-2\chi (C))/4=-1/2$. That $f_2$ is homogeneous of degree $-1/2$ can be proven similarly.

Next we analyze the exact sequence \ref{eq:exact}. When $p$ is large there is no contribution coming from $HF^{\mathrm{red}}$, so all the groups appearing in the sequence \ref{eq:exact} are isomorphic to either $\mathbb{Z}$ or $0$. When $p$ is large and  even $f_2=0$, so $f_1$ is an isomorphism. Similarly, when $p$ is large and odd $f_2$ is an isomorphism. By $U$--equivariance, $f_1$ sends $\mathcal T^+_{(0)}\subset HF^+(\partial W_n(k+1))$  isomorphically to $\mathcal T^+_{(-1/2)} \subset HF^+(S^1\times S^2)$. Similarly, $f_2$ induces an isomorphism between $\mathcal T^+_{(1/2)} \subset HF^+(S^1\times S^2)$ and $\mathcal T^+_{(0)}\subset HF^+(\partial W(k))$. Hence we have $\mathrm{coker}(f_2)\simeq HF^{\mathrm{red}}(W(k))$ and $\mathrm{ker}(f_1)\simeq HF^{\mathrm{red}}(W(k+1))$. By exactness, $f_3$ induces an isomorphism between $HF^{\mathrm{red}}(W(k))$ and $HF^{\mathrm{red}}(W(k+1))$.

Unlike the other two homomorphisms appearing in the exact triangle, $f_3$ may not be homogeneous. Even so, we claim that it preserves $\mathbb{Z}/2\mathbb{Z}$ grading.  When we attach a $2$--handle along $\gamma \subset \partial W(k)$ with framing $-1$ and call the resulting cobordism $X$, the Seifert surface of $\gamma$ gives rise to a closed surface $\Sigma \subset X$ whose homology class generates $H_2(X,\mathbb{Z})\simeq \mathbb{Z}$. Let $\mathfrak{s}_r$, $r \in \mathbb{Z}$, be the \spinc structure on $X(k)$ with $\langle c_1(\mathfrak{s}_r),[\Sigma]\rangle=2r+1$. Denote by $F^+_{X,\mathfrak{s}_r}$, the homomorphism associated to the \spinc structure $\mathfrak{s}_r$. It is known that, 

\begin{equation}\label{eqn:cobsum}
f_3=\sum_{r\in \mathbb{Z}} \pm F^+_{X,\mathfrak{s}_r}
\end{equation}

\noindent for some choice of signs. Each $F^+_{X,\mathfrak{s}_r}$ is homogeneous of degree $(c_1(\mathfrak{s}_r)^2-3\sigma(X)-2\chi (X))/4=-r(r+1))$ which is even for every $r$.

For the last part, suppose the genus of $\Sigma$ is one. Then the \spinc structure $\mathfrak{s}_r$ satisfies 

$$ \left | \langle c_1(\mathfrak{s}_r),[\Sigma]\rangle \right |+[\Sigma]\cdot[\Sigma] > 2g-2$$

\noindent for every $r\in \mathbb{Z}-\{-1, 0\}$, so it admits a non-trivial adjunction relation, \cite{OS7}.  Hence by equation \ref{eqn:cobsum}, $f_3$ can be written as 

\begin{equation}\label{eqn:deco}
f_3=\phi+U\cdot \psi
\end{equation}

\noindent for some $\phi$ and $\psi$, where $\phi$ is the portion of $f_3$ coming from  $\mathfrak{s}_{-1}$ and $\mathfrak{s}_0$. The degree calculation above shows that $\phi$ is homogeneous of degree $0$. We will show that $\phi$ restricts to an isomorphism between $HF^{\mathrm{red}}(\partial W(k))$ and $HF^{\mathrm{red}}(\partial W(k+1))$. Both of these groups are filtered as follows:

$$\mathrm{ker}(U) \subseteq \mathrm{ker}(U^2) \subseteq \cdots  \subseteq \mathrm{ker}(U^s)=HF^{\mathrm{red}}.$$

\noindent The isomorphism $f_3$ respects this filtration. In particular it induces an isomorphism on $\mathrm{ker}(U^t)/\mathrm{ker}(U^{t-1})$ for all $t=1,2,\cdots,s$. By equation \ref{eqn:deco}, the restriction of $f_3$ on each of these groups is the same as $\phi$. Having  seen that $\phi$ is an isomorphism on  each $\mathrm{ker}(U^t)/\mathrm{ker}(U^{t-1})$, we conclude that $\phi$ is an isomorphism on the whole $HF^{\mathrm{red}}(\partial W(k))$.

\end{proof}

Now we focus our attention to the family of Mazur type manifolds $W_n(k)$ of Figure \ref{fig:Cork}. In view of the previous proposition, one would like to have a genus one Seifert surface of the curve $\gamma \subset \partial W_n(k)$ for every $n$ and $k$. So far the authors neither constructed nor ruled out the existence of such a surface when  $n\neq 1$. The best they achive is the following:

\begin{prop}\label{prop:seif}
The curve $\gamma$ in Figure \ref{fig:Cork} bounds a surface of genus $n$ in $\partial W_n(k)$. 
\end{prop}

\begin{proof}
(Compare \cite{AK}).Replace the dotted circle with a $0$ framed $2$-handle in Figure \ref{fig:Cork}, this operation does not change the boundary. We see that $\partial W_n(k)$ is obtained by  doing surgery on a symmetric link $K_1 \cup K_2\subset S^3$ where the surgery coefficients are $k$ and $0$ respectively. For some choice of orientations we make sure that linking numbers  satisfy $\mathrm{lk}(\gamma,K_1)=-\mathrm{lk}(K_2,K_1)=1$. Therefore $\gamma$ and $K_2$ together bound an embedded surface $S$ in the complement of $K_1$. We apply Seifert's algorithm to see that the surface has genus $n$ . While doing $k$ framed surgery on $K_1$ has no effect on this surface, doing zero surgery caps off one of its boundary components. The resulting surface is what we have been looking for. An illustration of the Seifert surface is drawn in Figure \ref{fig:seifert}.
\end{proof}
\begin{figure}[h]
	\includegraphics[width=0.40\textwidth]{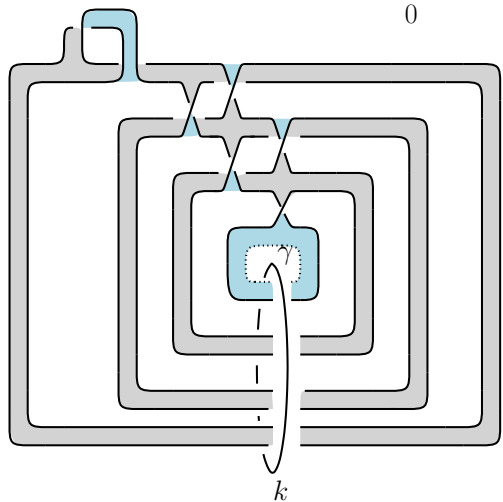}
		\caption{Seifert surface of $\gamma$ in $W_3(k)$}
	\label{fig:seifert}
\end{figure}

\begin{rem}
Combining the above results with \cite{AK} we see that when $n=1$,  

$$HF^+(-\Sigma(2,5,7))\simeq HF^+(-\Sigma(3,4,5))\simeq HF^+(-\Sigma(2,3,13))\simeq \mathcal{T}^+_{(0)}\oplus \mathbb{Z}_{(0)}\oplus \mathbb{Z}_{(0)}$$

\noindent This was previously observed in  \cite{AD} by similar tecniques, but applied in  a somewhat different way. 
\end{rem}

\begin{rem}
One can calculate the Seifert matrix from the Seifert surface in proposition \ref{prop:seif}.  From this, it is easy to see that the Alexander polynomial of $\gamma$ is trivial.
\end{rem}

\begin{theo}\label{handleplum}
For every $n=1,2,3,\cdots$, the $3$--manifold $\partial W_n(2n+1)$ bounds a $4$--manifold which is obtained by taking several disk bundles  over spheres and plumbing them together according to the almost rational graph indicated in Figure \ref{Plumb}.
\end{theo}

\begin{figure}[h]
	\includegraphics[width=0.50\textwidth]{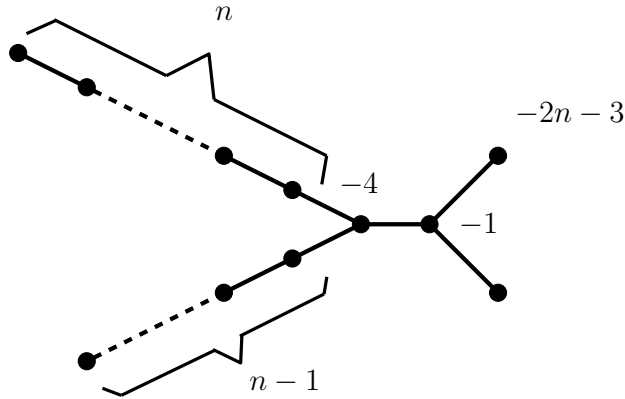}
	\caption{The plumbing configuration bounded by $\partial W_n(2n+1)$. Unlabeled vertices have weight $-2$.}
	\label{Plumb}
\end{figure}

\begin{proof}

The proof follows by a sequence of handlebody moves. They are indicated in Figure \ref{fig:hepsi} for $n=3$. The general case can be handled similarly.

\end{proof}

For any oriented $3$--manifold $Y$, let $-Y$ denote the the same $3$--manifold with the opposite orientation. It is easy to determine the effect of orientation reversal on Heegaard Floer homology, \cite{OS5}, \cite{OS2}. In  view of this fact, our theorem implies that the Heegaard Floer homology of the family of $3$-manifolds that we are interested in can be calculated combinatorially.

\begin{cor}
There is an algorithm calculating $HF^+(-\partial W_n(2n+1))$ for every $n$.
\end{cor}

\begin{proof}
Follows from result Nemethi \cite{N}(See also, Ozsv\'ath and Szab\'o \cite{OS1}) . The only thing that we should check is whether the plumbing graph $G$ is almost rational in the sense of \cite{N} Definition 8.1. Indeed by replacing the weight of the central vertex from $-1$ to $-3$ we get a rational graph (see section 6.2 of \cite{N}).
\end{proof}

We write a computer code in Magma that implements Nemethi's algorithm to calculate the Heegaard Floer homology of $-\partial W_n (2n+1)$. We list the Heegaard Floer homologies of fist three members of this family  in Corollary  \ref{Heeg} below to give the reader an idea of what these groups look like. See Table \ref{tab:tau} for the data that determines $HF^+(-\partial W_n(2n+1)$  for $n=1..7$. Our code is available online at \cite{K}.

\begin{cor}\label{Heeg}
Let $\mathbb{Z}^r_{(d)}$ be the quotient module $\mathbb{Z}[U]/U^r\mathbb{Z}[U]$ graded so that $U^{r-1}$ has degree $d$. Then we have the following identifications:
\begin{enumerate}
	\item $HF^+(-W_1(3))\simeq \mathcal T _0^+ \bigoplus (\mathbb{Z}_{(0)})^2$
	\item $HF^+(-W_2(5))\simeq \mathcal T _0^+  \bigoplus (\mathbb{Z}_{(0)})^4\oplus  (\mathbb{Z}_{(2)})^2\oplus (\mathbb{Z}_{(10)})^2$
	\item $HF^+(-W_3(7))\simeq \mathcal T _0^+  \bigoplus (\mathbb{Z}_{(0)})^4\oplus (\mathbb{Z}^2_{(0)})^2\oplus (\mathbb{Z}_{(2)})^2\oplus (\mathbb{Z}_{(4)})^2\oplus (\mathbb{Z}_{(12)})^2\oplus (\mathbb{Z}_{(14)})^2\oplus  (\mathbb{Z}_{(18)})^2\oplus (\mathbb{Z}_{(42)})^2$
\end{enumerate}
\end{cor}

Presumably, it is possible to find a closed formula for these Heegaard Floer homologies for every $n$ using Nemethi's techniques. Though the authors weren't able to find such a formula, one can easily verify that the rank of Heegaard Floer homology gets arbitrarily large with $n$ using a different method.

\begin{cor}\label{cor:rate}
$\mathrm{rank}HF^{\mathrm{red}}(-\partial W_n(k))=-\lambda(\partial W_n(k))=n(n+1)(n+2)/3$.
\end{cor} 

\begin{proof}

When $Y$ bounds an almost rational graph, $HF^+(-Y)$ is supported only in even degrees \cite{OS1}, implying $\chi(HF^{\mathrm{red}}(-Y)=\mathrm{rank}HF^{\mathrm{red}}(-Y)$.  Therefore by Theorem \ref{handleplum} and  Propositions \ref{indep} and \ref{prop:str}, it suffices to  show that $\lambda(-W_n(2n+1))=n(n+1)(n+2)/3$. This calculation is done in \cite{H}. Alternatively, this can be seen from the Casson invariant formula for plumbings given in section 2.4.2 of \cite{N}. Note that the formula involves calculation certain minors of the intersection matrix whose rank increases with $n$. One can get around that problem by realizing that each minor  actually corresponds to the order of the first homology of the three manifold obtained by deleting a vertex, and converting the resulting integral surgery diagram to a rational surgery diagram. The details are left to the reader (see the discussion in Section 5.1 of \cite{NN}).  
\end{proof}

\begin{rem}

It is known that every even integer can be realized as the Casson invariant of some homology sphere bounding a Mazur type manifold \cite{M}. However it is still an open question whether this could be achieved in the sub class of integral homology spheres which bound both a Mazur type manifold and an almost rational plumbing. 

\end{rem}
\newpage

\begin{figure}[h]
	\includegraphics[width=.80\textwidth]{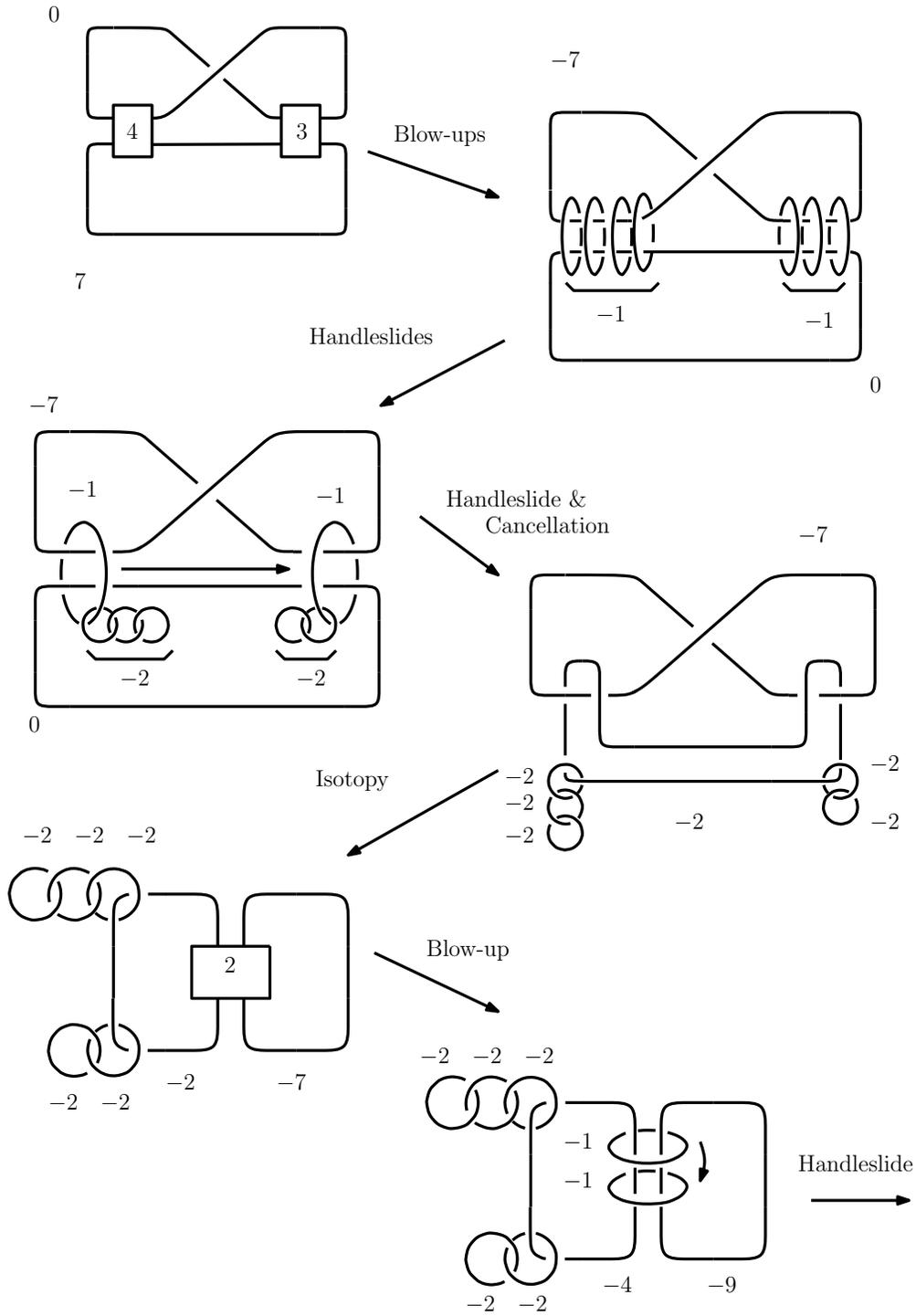}
	\caption{Handlebody moves showing that $\partial W_n(2n+1)$ bounds the plumbing indicated in Figure \ref{Plumb}.}
	\label{fig:hepsi}
\end{figure}

\newpage
 
\section{User's guide to Nemethi's Algorithm}

The aim of this section is to summarize the method of calculating the Heegaard Floer homology groups given in \cite{N}. This procedure is indicated schematically in Figure \ref{fig:Procedure}. In order to get the Heegaard Floer homology, one needs to calculate some intermediate objects, so--called computational sequence, tau function and graded root. We will define these objects and tell how they are used in the process in the following subsections. We are going to discuss only a special case where the three manifold is an integral homology sphere though the algorithm more generally works for rational homology spheres as well if one has an extra ingredient associated to \spinc structures.  

\begin{figure}[h]
	\includegraphics[width=0.80\textwidth]{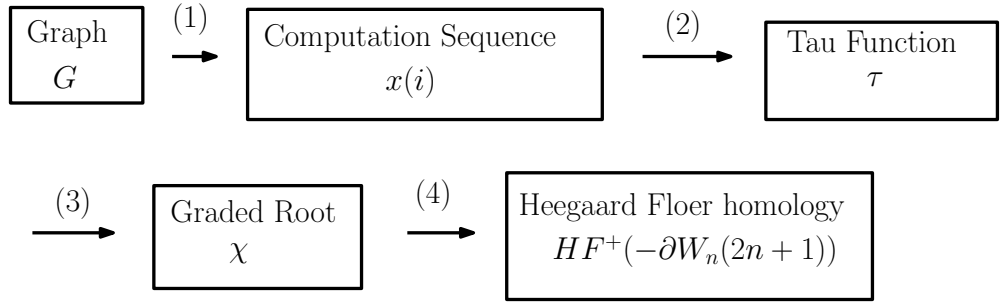}
	\caption{Steps of Nemethi's method}
	\label{fig:Procedure}
\end{figure}

\subsection{From graph to computation sequence:}\label{sec:alg}

Let $G_n$ be the plumbing graph indicated in Figure \ref{Plumb}. Denote its vertices by $\{v_0,v_1,\cdots, v_{2n+2}\}$. Choose the indices so that $v_0$ is the vertex with decoration $-1$. For any vertex $v$, let $m(v)$ and $\delta(v)$ denote the weight of $v$ and the number of edges connected to $v$ respectively. We distinguish the vertex $v_0$  because it is the only one satisfying the inequality $|m(v)|<\delta(v)$. Let $X(G_n)$ be the plumbed $4$--manifold. Each vertex $v$ represents a generator of $H_2(X(G_n))$ which will be denoted by the same letter. Let $(,)$ be the intersection form on $H_2(X(G_n))$. Then

$$(v_i,v_j)= \left \{
\begin{tabular}{ll}
$m(v_i)$ if $i=j$\\
1 if there is an edge connecting $v_i$ to $v_j$ \\
0 if there is no edge connecting $v_i$ to $v_j$.
\end{tabular} 
\right .$$

Inductively define a sequence $x(i)\in H_2(X(G_n))$ by first requiring that $x(0)=0$, and then obtaining  $x(i+1)$ from $x(i)$ by connecting them via a computation sequence $y_1,y_2,\cdots,y_l$ which is defined as follows
\begin{enumerate}
\item $y_1=x(i)+v_0$
\item Suppose $y_r$ is known. If $(y_r,v_j) > 0$ for some $j\neq 0$ then define $y_{r+1}:=y_r+v_j$
\item Repeat step (2) until you find $y_l$, such that  $(y_l,v_j)\leq 0$ for all $j\neq 0$ then declare that $x(i+1):=y_l$.
\end{enumerate}

Although $x(i)$ is an infinite sequence, only a finite portion of it will be used in the next step of the algorithm.

\subsection{From computation sequence to tau function:}

Using the output of the previous step, we inductively construct a function $\tau:\mathbb{N}\to \mathbb{Z}$:

\begin{itemize}
\item $\tau(0)=0$.
\item $\tau(i+1)=\tau(i)+1- (x(i),v_0)$.
\end{itemize}

It can be shown that $\tau$ is ultimately increasing. Indeed if $i_0$ is the first index with $\tau(i_0)=2$ then $\tau (i)\leq \tau (i+1)$ for all $i \geq i_0$. For the remaining part, we only need the restricton of $\tau$ to the set $\{1,2,3,\cdots,i_0\}$.

Regard $\tau$ as a finite sequence. Choose a  sub--sequence by first killing all the consecutive repetitions, and then taking all the local maxima and local minima of what  is left. We also drop the last entry which is equal to $2$. For example the tau function obtained from the graph $G_1$  is $[0,1,0,0,0,0,0,0,0,0,0,1,0,0,0,1,1,1,1,1,1,2]$ which reduces to $[0,1,0,1,0]$ after the reduction. We list the reduced tau functions of $G_n$ for $n=1..7$ in Table  \ref{tab:tau}. The data associated to larger values of $n$ will be available online soon.

\subsection{From tau function to graded root:} Given a natural number $i$, let $R_i$ denote the infinte tree  with the vertex set $\mathcal V :=\{ v^j: j\geq i\}$ and the edge set $\mathcal E:=\{[v^j,v^{j+1}]:j\geq i \}$. This tree also admits a grading function $\chi : \mathcal V\to \mathbb{Z}$, with $\chi (v^j)=j$. From the (reduced) tau function $\tau :\{0,1,2,\cdots, 2k\}$, we can construct an infinite tree

$$R_\tau:=R_{\tau (0)} \sqcup R_{\tau (1)}\sqcup R_{\tau (2)} \cdots \sqcup R_{\tau (2k)}/\sim,$$

\noindent where the equivalence relation $\sim$ is defined as follows:the reduced $\tau$ function attains its local maxima at odd integers.  Therefore $R_{\tau (2i+1)}$ naturally injects into both $R_{\tau (2i)}$, and $R_{\tau (2i+2)}$ in a way that preserves the grading. We identify $R_{\tau (2i+1)}$ with its images in $R_{\tau (2i)}$, and $R_{\tau (2i+2)}$ for every $i=0,1,2,..,k-1$. Then the grading functions on each $R_{\tau (j)}$ descend to a grading function $\chi_\tau$ on $R_\tau$.  The pair $(R_\tau,\chi _\tau)$ is called the graded root associated with $\tau$. An illustration of this construction is given in figure \ref{fig:ConGradedRoot}  when $\tau=[0,1,0,1,0]$. We draw the graded roots of the graphs $G_n$, for $n=1,2,3$ in Figure \ref{fig:GradedRoots}.

\begin{figure}[h]
	\includegraphics[width=0.50\textwidth]{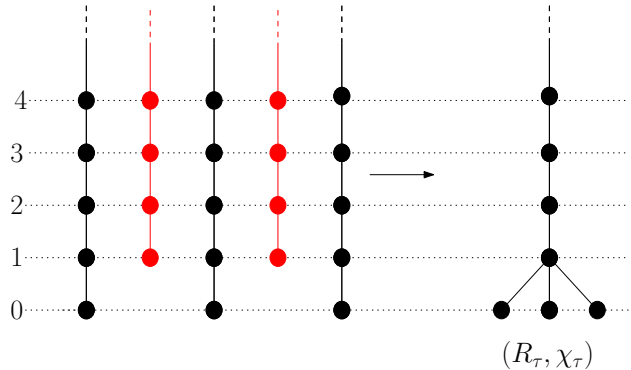}
	\caption{Graded Root of $G_1$.}
	\label{fig:ConGradedRoot}
\end{figure}

\subsection{From Graded root to Heegaard--Floer Homology:} Let $\mathcal V$ and $\mathcal E$ denote the vertex and edge sets of $R_\tau$. Define

$$ \mathbb{H}(R_\tau,\chi _\tau)=\{\phi:\mathcal V \to \mathcal T^+_{(0)} : U \cdot \phi (v)=\phi (w) \; \mathrm{whenever}\;  [v,w]\in \mathcal E \; \mathrm{and}  \; \chi_\tau (v) < \chi_\tau (w)\}$$

The set $ \mathbb{H}(R_\tau,\chi _\tau)$ naturally admits the structure of a $\mathbb{Z}[U]$-module. One can define a grading on it as follows: $\phi \in  \mathbb{H}(R_\tau,\chi _\tau)$ is said to be homogeneous of degree $d$ if $\phi (v)\in \mathcal T ^+ _{(0)}$ is homogeneous of degree $d-2\chi _\tau (v)$ for every $v\in \mathcal V$. Finally, the  Heegaard Floer homology is isomorphic to  $ \mathbb{H}(R_\tau,\chi _\tau)$ with a grade shift. In the cases of interest,  the exact amount of grade shift  can be found by requiring that $d(-\partial W_n(2n+1))=0$.
   
\tiny

\begin{table}[h]
\begin{tabular}{|l|l|}
\hline
$n=1$&$[ 0, 1, 0, 1, 0 ]$\\
\hline
$n=2$&$[ 0, 1, -4, -3, -5, -4, -5, -4, -5, -4, -5, -4, -5, -3, -4, 1, 0 ]$\\
\hline
		 &$[ 0, 1, -12, -11, -14, -13, -15, -14, -19, -18, -20, -19, -21,$ \\
$n=3$&$\; -19, -21, -20, -21, -20, -21, -20, -21, -20, -21, -19, -21,$\\
     &$\; -19, -20, -18, -19, -14, -15, -13, -14, -11, -12, 1, 0]$\\
\hline
		 &$[ 0, 1, -24, -23, -27, -26, -29, -28, -41, -40, -43, -42, -45,$\\
		 &$ -43, -46, -45,-47, -46, -51, -50, -52, -51, -53, -51, -54, -52,$\\
$n=4$&$ -54, -52, -54, -53, -54, -53, -54, -53, -54, -53, -54, -52, -54,$\\
		 &$ -52, -54, -51, -53, -51, -52, -50, -51, -46, -47, -45, -46, -43,$\\
		 &$ -45, -42, -43, -40, -41, -28, -29, -26, -27, -23, -24, 1, 0]$\\
\hline
		 &$[ 0, 1, -40, -39, -44, -43, -47, -46, -71, -70, -74, -73, -77, -75,$\\
		 &$ -79, -78,-81, -80, -93, -92, -95, -94, -97, -95, -99, -97, -100, -98,$\\
		 &$ -101, -100, -102,-101, -106, -105, -107, -106, -108, -106, -109, -107,$\\
		 &$ -110, -107, -110, -108,-110, -108, -110, -109, -110, -109, -110, -109,$\\
$n=5$&$ -110, -109, -110, -108, -110,-108, -110, -107, -110, -107, -109, -106,$\\
		 &$ -108, -106, -107, -105, -106, -101,-102, -100, -101, -98, -100, -97, -99,$\\
		 &$ -95, -97, -94, -95, -92, -93, -80, -81, -78, -79, -75, -77, -73, -74, -70,$\\
		 &$ -71, -46, -47, -43, -44, -39, -40, 1, 0 ]$\\
\hline
		&$[ 0, 1, -60, -59, -65, -64, -69, -68, -109, -108, -113, -112, -117, -115, -120,$\\
		&$-119, -123, -122, -147, -146, -150, -149, -153, -151, -156, -154, -158, -156,$\\
		&$-160, -159, -162, -161, -174, -173, -176, -175, -178, -176, -180, -178, -182,$\\
		&$-179, -183, -181, -184, -182, -185, -184, -186, -185, -190, -189, -191, -190,$\\ 
		&$-192, -190, -193, -191, -194, -191, -195, -192, -195, -192, -195, -193, -195,$\\
$n=6$&$-193, -195, -194, -195, -194, -195, -194, -195, -194, -195, -193, -195, -193,$\\ 
		&$-195, -192, -195, -192, -195, -191, -194, -191, -193, -190, -192, -190, -191,$\\ 
		&$-189, -190, -185, -186, -184, -185, -182, -184, -181, -183, -179, -182, -178,$\\ 
		&$-180, -176, -178, -175, -176, -173, -174, -161, -162, -159, -160, -156, -158,$\\ 
		&$-154, -156, -151, -153, -149, -150, -146, -147, -122, -123, -119, -120, -115,$\\ 
		&$-117, -112, -113, -108, -109, -68, -69, -64, -65, -59, -60, 1, 0 ]$\\
\hline
		&$[ 0, 1, -84, -83, -90, -89, -95, -94, -155, -154, -160, -159, -165, -163, -169,$\\ 
		&$-168, -173, -172, -213, -212, -217, -216, -221, -219, -225, -223, -228, -226,$\\ 
		&$-231, -230, -234, -233, -258, -257, -261, -260, -264, -262, -267, -265, -270,$\\ 
		&$-267, -272, -270, -274, -272, -276, -275, -278, -277, -290, -289, -292, -291,$\\ 
		&$-294, -292, -296, -294, -298, -295, -300, -297, -301, -298, -302, -300, -303,$\\ 
		&$-301, -304, -303, -305, -304, -309, -308, -310, -309, -311, -309, -312, -310,$\\ 
$n=7$&$-313, -310, -314, -311, -315, -311, -315, -312, -315, -312, -315, -313, -315,$\\ 
		&$-313, -315, -314, -315, -314, -315, -314, -315, -314, -315, -313, -315, -313,$\\ 
		&$-315, -312, -315, -312, -315, -311, -315, -311, -314, -310, -313, -310, -312,$\\ 
		&$-309, -311, -309, -310, -308, -309, -304, -305, -303, -304, -301, -303, -300,$\\ 
		&$-302, -298, -301, -297, -300, -295, -298, -294, -296, -292, -294, -291, -292,$\\ 
		&$-289, -290, -277, -278, -275, -276, -272, -274, -270, -272, -267, -270, -265,$\\ 
		&$-267, -262, -264, -260, -261, -257, -258, -233, -234, -230, -231, -226, -228,$\\ 
		&$-223, -225, -219, -221, -216, -217, -212, -213, -172, -173, -168, -169, -163,$\\ 
		&$-165, -159, -160, -154, -155, -94, -95, -89, -90, -83, -84, 1, 0]$\\
\hline
\end{tabular}
\caption{Reduced $\tau$ functions associated to the plumbings indicated in Figure \ref{Plumb} for $n=1,2,\cdots,7$.}
\label{tab:tau}
\end{table}

\normalsize

\begin{figure}[h]
	\includegraphics[width=0.70\textwidth]{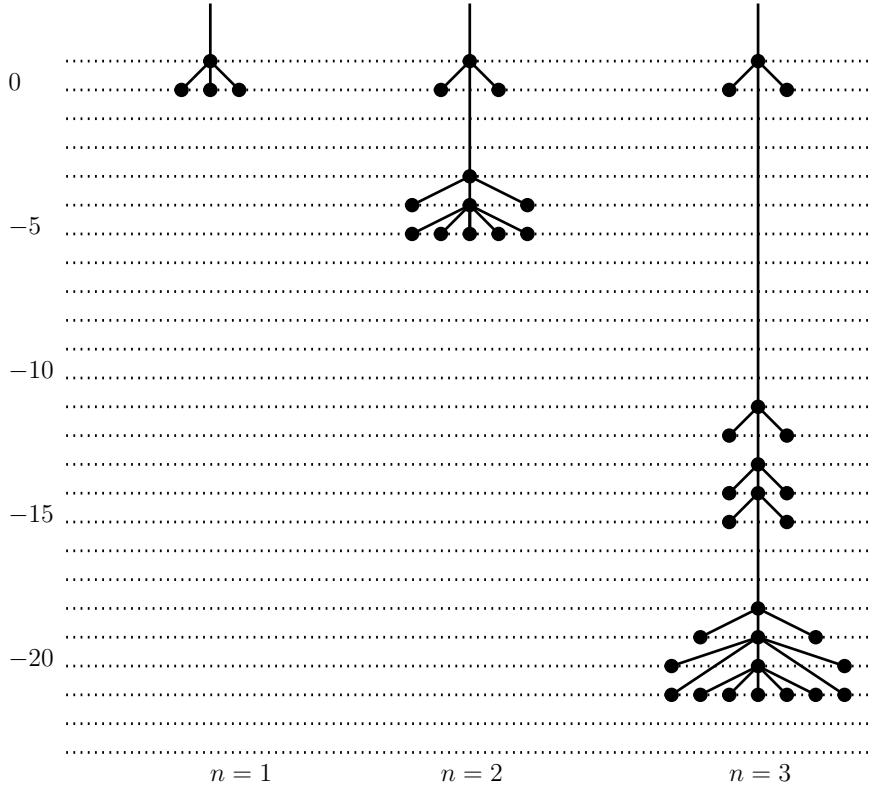}
	\label{fig:GradedRoots}
	\caption{The graded roots generated by $\tau$ functions in Table \ref{tab:tau}, for $n=1,2,3$.}
\end{figure}

\section{Some Brieskorn Spheres Bounding Contractible Manifolds}\label{sec:Bries}
Thus far we observed that if a Mazur type manifold has the same boundary as a plumbing up to change of framings then its Ozsv\'ath--Szab\'o invariant can be calculated combinatorially. We have seen this happenning in an infinite family of contractible manifolds. In \cite{CH}, Casson and Harer constructed many examples of this sort. Our goal in this section is to study Heegaard Floer homologies of their examples, and to make a comparison with the results of the calculations that we made in the previous sections.

Given pairwise relatively prime natural numbers $p,q,r$, the Brieskorn manifold $\Sigma (p,q,r)$ is defined as the intersection of the complex hypersurface $z_1^p+z_2^q+z_3^r=0$ with a five dimensional sphere centered at the origin with small radius.

\begin{prop}The following Brieskorn manifolds bound Mazur type manifolds.
\begin{enumerate}
	\item $\Sigma (p,ps\pm 1, ps \pm 2)$, $p$ odd.
	\item $\Sigma (p, ps -1, ps+1)$, $p$ even $s$ odd.
\end{enumerate}
\end{prop}
\begin{proof}
See \cite{NW}
\end{proof}

\begin{rem}
Most of these Mazur type manifolds do not  give us  corks even if we are allowed to  change the framing of the $2$--handle. This is because the their two handles are knotted, so there is no obvious involution on the boundary.
\end{rem}

 Any Brieskorn manifold bounds an almost rational plumbing. Indeed, let $e_0,p',q',r'$ be the unique integers satisfying $0<p'<p$, $0<q'<q$, $0<r'<r$, and

$$e_0pqr+p'qr+pq'r+pqr'=-1.$$

\noindent Then $\Sigma (p,q,r)$ is the Seifert fibred space with Seifert invariants $(e_0,p,p',q,q',r,r')$, and as such it bounds a star shaped plumbing with three branches. The central vertex has weight $e_0$, and replacing this weight with $-3$ makes the graph rational.

One consequence of the fact that Brieskorn manifolds bound almost rational plumbings is that their Heegaard Floer homology can be calculated combinatorially. In fact Nemethi gives an explicit formula for their (unreduced) tau function, which determines the Heegaard Floer homology, in terms of the Seifert invariants, \cite{N}. This formula is in general somewhat difficult to handle in order to say something about Heegaard Floer homology of infinite families of manifolds,  a couple of remarkable exceptions being \cite{BN} and \cite{N2}.

\tiny
\begin{table}[h]
\begin{tabular}{|c|c|c|}
\hline
& &\\
Brieskorn manifold&$\Sigma (p, ps\pm 1,ps\pm 2)$& $\Sigma (p, ps- 1,ps+ 1)$\\
$\Sigma (p,q,r)$& $p$ odd & $p$ even, $s$ odd\\
& &\\
\hline
& &\\
Seifert Invariants & $\displaystyle(-1,p,\frac{p-1}{2},ps \pm 1,s,ps \pm 2,\frac{ps-s\pm 2}{2}) $ & $\displaystyle(-1,p,1,ps - 1,\frac{ps-s -2}{2},ps+1,\frac{ps-s +2}{2}) $\\
$(e_0,p,p',q,q',r,r')$& $p$ odd& $p$ even, $s$ odd\\
& &\\
\hline
& &\\
tau function& $\displaystyle \sum_{j=0}^{n-1}1+j-\left \lceil \frac{j(p-s)}{2p} \right \rceil -  \left\lceil \frac{js}{ps\pm 1} \right\rceil- \left\lceil \frac{j(ps-s\pm 2)}{2ps \pm 4} \right\rceil $ &$\displaystyle \sum_{j=0}^{n-1}1+j-\left \lceil \frac{j}{p} \right \rceil -  \left \lceil \frac{j(ps-s-1)}{2ps- 2} \right\rceil- \left\lceil \frac{j(ps-s+ 1)}{2ps +2} \right\rceil $\\
$\tau (n)$ & & \\
& &\\
\hline
& &\\
$\mathrm{rank}(HF^{\mathrm{red}}(-\Sigma))$&$\displaystyle \frac{s(p^2-1)(ps\pm 3)}{24}$& $\displaystyle \frac{p^3s^2}{24}-\frac{ps^2}{24}- \frac{p}{8}$\\
$=-\lambda (\Sigma)$& &\\
& &\\
\hline
\end{tabular}
\caption{.}
\label{tab:bries}
\end{table}

\normalsize

We can calculate the  rank of Casson Harer manifolds, without appealing to Nemethi's formula, by proceeding as in the proof of Corollary \ref{cor:rate}. Since these manifolds bound both a contractible manifold and an almost rational plumbing, their correction term is zero and their Heegaard Floer homology is supported only in even degrees. Therefore, it suffices to calculate the Casson invariants. In \cite{NW} Lemma 1.5, an explicit formula of Casson invariants of Brieskorn manifolds is given. The formula involves certain Dedekind sums, which could be calculated by repeatedly applying the  reciprocity law. In table \ref{tab:bries} we listed all the above mentioned invariants of these manifolds. One interesting feature in this table is that the rank of the Heegaard Floer homology increases in the order of $p^3$, which is somehow consistent with what we observed in Corollary \ref{cor:rate}.

 \textit{Acknowledgements:} The final part of this work was completed while both authors were participating in the FRG Workshop in Miami. We would like to thank N. Saveliev for organizing this workshop. The second author would like to thank R. Gompf for bringing \cite{CH} to his attention, and P. Ozsv\'ath for illuminating discussions.

\end{document}